\documentclass[11pt]{article}
\usepackage{amsthm}
\usepackage{amsfonts}
\usepackage{amsmath}
\usepackage{colonequals}
\usepackage{epsfig}
\usepackage{url}
\newcommand{\GG}{{\cal G}}
\newcommand{\CC}{{\cal C}}
\newcommand{\PP}{{\cal P}}
\newcommand{\RR}{{\cal R}}
\newtheorem{conjecture}{Conjecture}
\newtheorem{theorem}{Theorem}
\newtheorem{corollary}[theorem]{Corollary}
\newtheorem{lemma}[theorem]{Lemma}

\newtheorem{observation}[theorem]{Observation}
\newcommand{\tw}{\mbox{tw}}
\newcommand{\Tw}{\mbox{Tw}}
\newcommand{\Vs}{\mbox{Vs}}
\newcommand{\Be}{\mbox{Be}}
\newcommand{\Ex}{\mbox{Ex}}

\title{Sublinear separators, fragility and subexponential expansion}
\author{Zden\v{e}k Dvo\v{r}\'ak\thanks{Computer Science Institute, Charles University, Prague, Czech Republic. E-mail: {\tt rakdver@iuuk.mff.cuni.cz}.
Supported by project GA14-19503S (Graph coloring and structure) of Czech Science Foundation 
and by project LH12095 (New combinatorial algorithms - decompositions, parameterization, efficient solutions) of Czech Ministry of Education.}}
\date{}

\begin{document}
\maketitle

\begin{abstract}
Let $\GG$ be a subgraph-closed graph class with bounded maximum degree.
We show that if $\GG$ has balanced separators whose size is smaller than linear by a polynomial factor,
then $\GG$ has subexponential expansion. This gives a partial converse
to a result of Ne\v{s}et\v{r}il and Ossona de Mendez.  As an intermediate step, the proof uses a new kind
of graph decompositions.
\end{abstract}

The concept of graph classes with bounded expansion was introduced by Ne\v{s}et\v{r}il and Ossona de Mendez~\cite{grad1}
as a way of formalizing the notion of sparse graph classes.  Let us give a few definitions.

For a graph $G$, a \emph{$k$-minor of $G$} is any graph obtained from $G$ by contracting
pairwise vertex-disjoint subgraphs of radius at most $k$ and removing vertices and edges.  Thus, a $0$-minor is just a subgraph of $G$.
Let us define $\nabla_k(G)$ as $$\max\left\{\frac{|E(G')|}{|V(G')|}:\mbox{$G'$ is a $k$-minor of $G$}\right\}.$$
For a class $\GG$, let $\nabla_k(\GG)$ be the supremum of $\nabla_k(G)$ for $G\in \GG$ (or $\infty$ if $\nabla_k$ is unbounded for the graphs
in the class).  If $\nabla_k(\GG)$ is finite for every $k\ge 0$, we say that $\GG$ has \emph{bounded expansion};
and if $f$ is a function such that $f(k)\ge\nabla_k(\GG)$ for every $k\ge 0$, we say that $f$ \emph{bounds the expansion of $\GG$}.
If $\lim_{k\to\infty} \frac{\log \nabla_k(\GG)}{k}=0$, we say that $\GG$ has \emph{subexponential expansion}.

The definition is quite general---examples of classes of graphs with bounded expansion include proper minor-closed classes of graphs,
classes of graphs with bounded maximum degree, classes of graphs excluding
a subdivision of a fixed graph, classes of graphs that can be embedded
on a fixed surface with bounded number of crossings per edge and
many others, see~\cite{osmenwood}.

Importantly, bounded expansion also implies a wide range of
interesting structural and algorithmic properties, generalizing many results from proper minor-closed classes of graphs.
For example, graphs in any class with bounded expansion have bounded chromatic number, acyclic chromatic number,
star chromatic number, and other generalized variants of the chromatic number~\cite{grad1}.
For graphs from such a class, there exists a linear-time algorithm to test the presence of a fixed subgraph~\cite{grad2}
(as the subgraph testing problem is $W[1]$-complete when parameterized by the subgraph~\cite{fellows}, such an
algorithm is unlikely to exist for all graphs).  This algorithm was further generalized to testing any property
expressible in the first order logic~\cite{dvorak2013testing}.
Other related results include bounds on the growth function of classes with bounded expansion~\cite{smallcl},
and parameterized algorithmic results on induced matchings~\cite{induced} and dominating sets~\cite{domker}.
For a more in-depth introduction to the topic, the reader is referred to the book of Ne\v{s}et\v{r}il and Ossona de Mendez~\cite{nesbook}.

The bounds and the time complexity of the algorithms we mentioned in the previous paragraph of course depend on the function bounding the expansion
of the class; hence, it would be useful to be able to estimate this function for a given graph class.
However, while there is an extensive theory for qualitatively deciding whether a class of graphs has bounded
expansion~\cite{subdivchar,osmenwood}, we only know a tight estimate for the function bounding the expansion
for a few special classes of graphs (proper minor-closed classes, and the class of graphs with given maximum degree).

One of the properties of graph classes with bounded expansion that might lead to improving the estimates
is a connection to small balanced separators.
A \emph{separation} of a graph $G$ is a pair $(A,B)$ of edge-disjoint subgraphs of $G$ such that $A\cup B=G$,
and the \emph{size} of the separation is $|V(A)\cap V(B)|$.  Observe that $G$ has no edge with one end with $V(A)\setminus V(B)$
and the other end in $V(B)\setminus V(A)$, and thus the set $V(A)\cap V(B)$ separates $V(A)\setminus V(B)$ from
$V(B)\setminus V(A)$ in $G$.
A separation $(A,B)$ is \emph{balanced} if $|V(A)\setminus V(B)|\le 2|V(G)|/3$ and $|V(B)\setminus V(A)|\le 2|V(G)|/3$.
Note that $(G,G-E(G))$ is a balanced separation.
For a graph class $\CC$, let $s_\CC(n)$ denote the smallest nonnegative integer such that every graph in $\CC$ with at most $n$ vertices
has a balanced separation of size at most $s_\CC(n)$.  We say that $\CC$ has \emph{sublinear separators} if
$\lim_{n\to\infty} \frac{s_\CC(n)}{n}=0$, and that $\CC$ has \emph{strongly sublinear separators} if
there exist constants $c\ge 1$ and $0\le \delta<1$ such that $s_\CC(n)\le cn^\delta$ for every $n\ge 0$.

Lipton and Tarjan brought focus on the notion of sublinear separators by showing in~\cite{lt79} that the class $\CC_p$ of planar
graphs satisfies $s_{\CC_p}(n)=O(\sqrt{n})$, and by pointing out that sublinear separators
lead to a natural divide-and-conquer approach, useful especially in the design of efficient polynomial-time algorithms,
as well as of approximation algorithms and of exact algorithms with subexponential time complexity~\cite{lt80}. Since then, numerous similar applications were
found~\cite{ssep1,ssep2,ssep3,ssep4,ssep5,ssep6,ssep7}, establishing the importance of the concept.

Later, it was shown that graphs embedded on other surfaces~\cite{gilbert} and all proper minor-closed
graph classes~\cite{kreedsep} also have strongly sublinear separators. 
Building upon the previous result of Plotkin, Rao and Smith~\cite{plotkin},
Ne\v{s}et\v{r}il and Ossona de Mendez~\cite{grad2} made the following observation linking balanced separators
to bounded expansion, which qualitatively generalizes all the previous results (let us remark that
the classes studied in~\cite{lt79,gilbert,kreedsep} all have expansion bounded by a constant function).

\begin{theorem}[{Ne\v{s}et\v{r}il and Ossona de Mendez~\cite[Theorem 8.3]{grad2}}]\label{thm-subexpsep}
Every graph class with subexponential expansion has sublinear separators.
\end{theorem}

Theorem~\ref{thm-subexpsep} can be used to establish a lower bound on the expansion function of a class,
and it cannot be significantly improved, since 3-regular expanders have expansion bounded by $f(k)=2^k$ and do not have sublinear separators.
In this paper, we indicate that Theorem~\ref{thm-subexpsep} might actually be an almost precise
characterization of classes of graphs with sublinear separators (or, alternatively, of classes of graphs with subexponential
expansion) by proving its weak converse.

\begin{theorem}\label{thm-main}
Let $\GG$ be a subgraph-closed class of graphs with bounded maximum degree.
If $\GG$ has strongly sublinear separators, then there exists $\gamma\ge 0$ such that the expansion of $\GG$ is bounded by $f(k)=\gamma e^{k^{3/4}}$.
Hence, $\GG$ has subexponential expansion.
\end{theorem}

Theorem~\ref{thm-main} is the first general criterion implying subexponential expansion that has been found so far,
and indeed, one of the first results giving a reasonably small upper bound on the expansion function of a class of graphs.
The assumption that $\GG$ is subgraph-closed is natural, excluding dense graphs with balanced separators (such as two
cliques of the same size).  Unlike the outcome of Theorem~\ref{thm-subexpsep}, we require strongly sublinear separators; however,
this stronger assumption holds in most natural examples of graph classes known to have sublinear separators.
Also, such an assumption cannot be avoided entirely: consider for example the class $\GG$ consisting of all graphs $G$
such that the distance in $G$ between any two vertices of degree at least $3$ is at least $\log |V(G)|$.
The class $\GG$ satisfies $s_\GG(n)=O(n/\log n)$, but it has exponential expansion.

The major flaw in Theorem~\ref{thm-main} is the assumption on bounded maximum degree, which severely restricts its applicability.
While it is required in the proof, I have no reason to believe that it should be necessary and I propose the following conjecture.

\begin{conjecture}
Every subgraph-closed class of graphs with strongly sublinear separators has subexponential expansion.
\end{conjecture}

The proof of Theorem~\ref{thm-main} proceeds by contradiction, showing that a class with (nearly) exponentially large expansion
cannot have strongly sublinear separators.  The proof has two main ingredients.  Firstly, it is relatively easy to deal with the situation when for arbitrarily large $n$,
$\GG$ contains an $n$-vertex graph $G$ such that $\nabla_{\log n}(G)\ge n^\varepsilon$ for some $\varepsilon>0$, and to show that $G$ has a subgraph without sufficiently
small balanced separation.  This is based on (a generalization of) the following theorem on shallow clique minors, which is of a separate interest
in the context of previous results on the topic~\cite{kopyb,jiang}.
\begin{theorem}\label{thm-lacli}
For every $\varepsilon$ such that $0<\varepsilon\le 1$, there exist integers $n_0,d\ge 0$ such that
if a graph $G$ on $n\ge n_0$ vertices has at least $n^{1+\varepsilon}$ edges, then it contains
$K_{\lfloor n^{\varepsilon/6}\rfloor}$ as a $d$-minor.
\end{theorem}
A variant of Theorem~\ref{thm-lacli} was proved in my dissertation thesis~\cite{dvorak}, and we give a somewhat simplified version of the proof in
Section~\ref{sec-small}.  This part of the argument does not require bounded maximum degree.

It remains to consider the case that $\nabla_{\log n}(G)\ll n^\varepsilon$, and thus for a large $k\ge 0$, the graph $G$ showing that $\nabla_k(\GG)$ is nearly exponential in $k$ has
many vertices (compared to any exponential in $k$).  We would like to split $G$ to components whose size makes it possible to
apply the result of the previous paragraph, by removing a small part of $G$.  Lipton and Tarjan~\cite{lt80} show that if $\GG$ has strongly sublinear
separators, then we can remove some set $S$ of vertices of $G$ of sublinear size so that each component of $G-S$ has bounded size.
However, we cannot directly apply this result, since some important part of the subgraph of $G$ determining $\nabla_k(G)$ could be contained in $S$.
Hence, as the second main ingredient, we need a strengthening of the result, showing that there exist many possible ``almost disjoint'' choices for $S$.

\begin{lemma}\label{lemma-split}
Let $\GG$ be a subgraph-closed class of graphs, with bounded maximum degree and strongly sublinear separators.
There exists $b>1$ with the following property.  For every $\varepsilon>0$ and every $G\in\GG$, there
exists some $m\ge 0$ and (not necessarily distinct) sets $S_1, \ldots, S_m\subseteq V(G)$ such that
each vertex of $G$ is contained in at most $\varepsilon m$ of these sets and such that for $1\le i\le m$,
each component of $G-S_i$ has at most $b^{1/\varepsilon}$ vertices.
\end{lemma}

The proof of Lemma~\ref{lemma-split} is given in Section~\ref{sec-split}, and the results are combined to a proof of Theorem~\ref{thm-main}
in Section~\ref{sec-subexpexp}.  The notion of existence of a large number of almost disjoint subsets
whose removal ensures some property can be viewed as a fractional version of a certain previously studied concept;
we establish this connection in Section~\ref{sec-fragility}.
The notion may be of separate interest because of its potential algorithmic applications.  Although tangential to the topic of the paper,
we give more details in Section~\ref{sec-fragilityap}.

\section{Probabilistic results}\label{sec-probab}

We will use several tools from the probability
theory, which we quickly recall here; for a more in-depth treatment see e.g.~\cite{mat}.

A \emph{finite probability space} is a finite set $S$ together with a \emph{probability distribution} $\tau:S\to [0,1]$
such that $\sum_{s\in S} \tau(s)=1$.  An \emph{event} $T$ is a subset of $S$, and its probability $\text{Prob}(T)$ is
$\sum_{s\in T}\tau(s)$.  For a unary predicate $\varphi$, we write $\text{Prob}[\varphi]$ as a shortcut for
$\text{Prob}(\{s\in S:\varphi(s)\})$.  A \emph{random variable} is any function $X:S\to\mathbf{R}$, and its \emph{expected value}
is $E(X)=\sum_{s\in S} \tau(s)X(s)$.

We use several basic inequalities, such as Markov's inequality (see~\cite{mat}, Lemma~4.0.2).
\begin{lemma}\label{lemma-markov}
Let $X$ be a non-negative random variable.  For any positive real number $r$,
$$\text{Prob}[X\ge r]\le E(X)/r$$ and if $E(X)>0$, then
$$\text{Prob}[X>r]<E(X)/r.$$
\end{lemma}

Random variables $X_1$, \ldots, $X_n$ are \emph{independent} if for all measurable sets $A_1, \ldots, A_n\subseteq\mathbf{R}$,
we have $$\text{Prob}[X_1\in A_1, \ldots, X_n\in A_n]=\text{Prob}[X_1\in A_1]\text{Prob}[X_2\in A_2]\ldots\text{Prob}[X_n\in A_n].$$
We need the following corollary of Chernoff's bound (see~\cite{mat}, Theorem~7.2.1).

\begin{lemma}\label{lem-chernoff}
Let $X_1$, \ldots, $X_n$ be independent random variables, each of them
attaining value $1$ with probability $p$, and having value $0$ otherwise.
Let $X=X_1+\ldots+X_n$.  Then
$$\text{Prob}\bigl[\,X\le np/2\,\bigr] < \exp\Bigl(-\frac{3np}{28}\Bigr).$$
Furthermore, for any real number $r\ge n$,
$$\text{Prob}\bigl[\,X\ge 2rp\,\bigr] < \exp\Bigl(-\frac{3rp}{8}\Bigr).$$
\end{lemma}

\section{Fragility and fractional fragility}\label{sec-fragility}

A \emph{tree decomposition} $(T,\beta)$ of a graph $G$ is a tree $T$ and a function $\beta:V(T)\to 2^{V(G)}$ assigning a \emph{bag} $\beta(u)\subseteq V(G)$ to
each vertex $u\in V(T)$, such that
\begin{itemize}
\item for every $v\in V(G)$, there exists $u\in V(T)$ with $v\in\beta(u)$,
\item for every $vw\in E(G)$, there exists $u\in V(T)$ with $\{v,w\}\subseteq\beta(u)$, and
\item for every $v\in V(G)$, the set $\{u:v\in\beta(u)\}$ induces a connected subtree of $T$.
\end{itemize}
The \emph{width} of the decomposition is
the maximum of the sizes of its bags minus one, and the treewidth of $G$ is the minimum of the widths of its tree decompositions.

Consider a connected planar graph $G$.  For each vertex $v\in V(G)$ and an integer $r>0$, the subgraph of $G$ induced by the vertices at distance
at most $r$ from $v$ has treewidth at most $3r+1$, as shown by Robertson and Seymour~\cite{rs3}.
For integers $k\ge 1$ and $0\le t\le k-1$, let $Z_{t,k}$ denote the set of all vertices whose distance from $v$ is congruent to $t$ modulo $k$.
As a corollary of the preceding observation, the graph $G-Z_{t,k}$ has treewidth at most $3k+1$, see~\cite{eppstein00} for more details.  This observation
is very useful in the design of approximation algorithms, as for any set $X\subseteq V(G)$ (e.g., an optimal solution to an optimization problem), there exists
$t$ such that $|Z_{t,k}\cap X|\le |X|/k$.  Thus, it may be possible to find an optimal solution to a problem in $G-Z_{k,t}$ using its bounded treewidth,
then extend it to a near-optimal solution in $G$.  See Baker~\cite{baker1994approximation} for several algorithms along these lines.

Of course, instead of bounded treewidth, we could require any other property useful for the design of algorithms.  This motivates the following definitions.
A \emph{class property} $\PP$ is a class of graph classes.
For instance,
\begin{itemize}
\item let $\Tw$ denote the class property consisting of all graph classes with bounded treewidth;
\item let $\Be$ denote the class property consisting of all classes with bounded expansion; and,
\item let $\Vs$ denote the class property consisting of all graph classes with bounded component size,
where a class $\CC$ has \emph{bounded component size} if there exists some $t\ge 0$ such that every connected component
of a graph in $\CC$ has at most $t$ vertices.
\end{itemize}
Note that $\Vs\subset\Tw\subset\Be$.

Let $G$ be a graph and $\CC$ a class of graphs.  A \emph{packing} in $G$ is a multiset of pairwise vertex-disjoint 
subsets of $G$ (and thus only the empty set can appear multiple times in the packing).
A packing $P$ in $G$ is \emph{$\CC$-complementary} if for every $X\in P$, the graph $G-X$ belongs to $\CC$.
A class of graphs $\GG$ is \emph{$\PP$-fragile} if for every $k\ge 1$, there exists a class $\CC\in\PP$ such that
every graph in $\GG$ has a $\CC$-complementary packing of size $k$ (let us remark that the choice of $\CC$ is not necessarily unique).
For a given integer $k\ge 1$, we say that a class $\CC$ with this property is a \emph{$(1/k)$-witness of the $\PP$-fragility of $\GG$}.  We use $1/k$ rather than $k$ for consistency
with a notation we will introduce in a few paragraphs.  

As we already outlined, the most studied version of fragility deals with treewidth, and
the example we started with can be stated as the claim that the class of all planar graphs is $\Tw$-fragile.
One of the most general results in the area is by DeVos et al.~\cite{devospart}, showing that every proper minor-closed class of graphs is $\Tw$-fragile.
Let us also remark similar concepts based on edge removal~\cite{devospart} or edge contraction~\cite{contrpart}.
Despite these encouraging results, it turns out that even very simple and well-structured classes of graphs need not be
$\Tw$-fragile.

Let $R_n$ be the strong product of three paths with $n$ vertices, that is, the graph with vertex set
$\{(i,j,k):1\le i,j,k\le n\}$ such that two distinct vertices $(i_1,j_1,k_1)$ and $(i_2,j_2,k_2)$ are adjacent iff
$|i_1-i_2|\le 1$, $|j_1-j_2|\le 1$ and $|k_1-k_2|\le 1$; the graph $R_4$ is depicted in Figure~\ref{fig-grid} (the thickness
of edges is just to aid the visualization).  Let $\RR=\{R_n:n\ge 1\}$.

\begin{figure}
\begin{center}
\includegraphics[width=60mm]{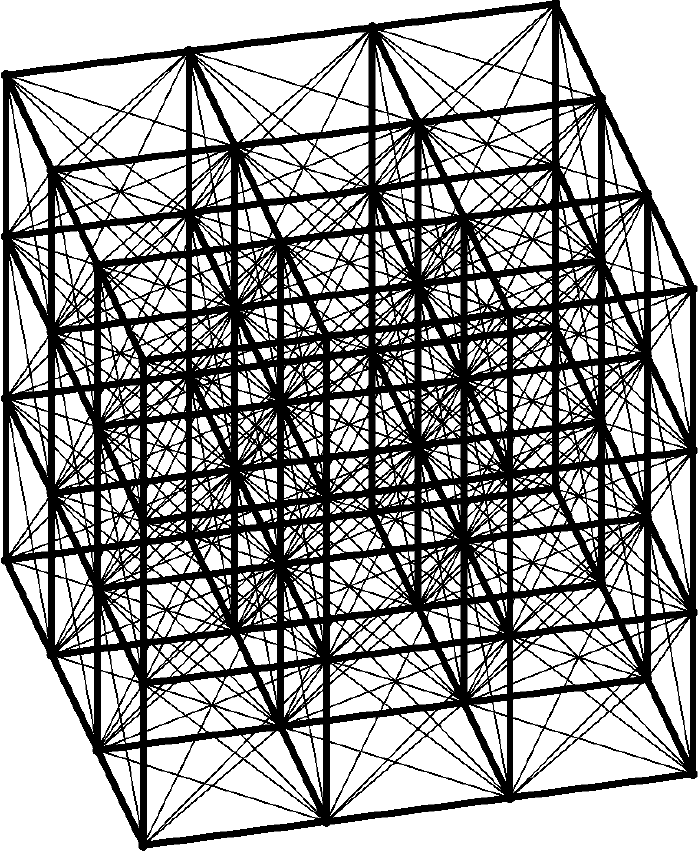}
\end{center}
\caption{The graph $R_4$.}\label{fig-grid}
\end{figure}

Berger et al.~\cite{gridtw} proved that for every $k\ge 0$, there exists $n\ge 1$ such that
for every partition $A$, $B$ of the vertices of $R_n$, either $R_n[A]$ or $R_n[B]$ has
treewidth at least $k$. Hence, we get the following.

\begin{theorem}[Berger et al.~\cite{gridtw}]\label{thm-gridtw}
The class $\RR$ is not $\Tw$-fragile.
\end{theorem}

This is problematic in our intended application, since the class $\RR$ has strongly sublinear separators.
To overcome this issue, we introduce a fractional relaxation of fragility.
Given a graph $G$ and a class $\CC$, let $G-\CC=\{X\subseteq V(G): G-X\in \CC\}$.
A \emph{fractional $\CC$-complementary packing} is an assignment $\pi:G-\CC\to [0,1]$
such that $\sum_{X\in G-\CC} \pi(X)=1$.
The \emph{thickness} of the fractional packing is
$$\max\left\{\sum_{X\in G-\CC, v\in X} \pi(X): v\in V(G)\right\}.$$
Let us remark that if $G\in\CC$, then $G$ has a fractional $\CC$-complementary packing of thickness $0$ obtained by setting $\pi(\emptyset)=1$
and $\pi(X)=0$ for every nonempty $X\subseteq V(G)$.
A convenient way how to view a fractional $\CC$-complementary packing of thickness $\varepsilon$ is as a probability distribution on $G-\CC$
such that for every vertex $v$, the probability that $v$ belongs to a set chosen at random according to this distribution is at most $\varepsilon$.

A class of graphs $\GG$ is \emph{fractionally $\PP$-fragile} if for every $\varepsilon>0$, there exists a class $\CC\in\PP$
such that each graph in $\GG$ has a fractional $\CC$-complementary packing of thickness at most $\varepsilon$.
For a given $\varepsilon>0$, we say that such a class $\CC$ is an \emph{$\varepsilon$-witness of the fractional $\PP$-fragility of $\GG$}.

Clearly, if a class is $\PP$-fragile, it is also fractionally $\PP$-fragile.
On the other hand, the following example together with Theorem~\ref{thm-gridtw} shows that fractional $\Tw$-fragility (or even fractional $\Vs$-fragility) does not imply $\Tw$-fragility.

\begin{lemma}\label{lemma-gridfract}
The class $\RR$ is fractionally $\Vs$-fragile.
\end{lemma}
\begin{proof}
Consider any $\varepsilon>0$, and let $u=\lceil 3/\varepsilon\rceil$.  Let $\CC\in \Vs$ be the class of graphs in that every component has at most $(u-1)^3$ vertices.

Consider a graph $R_n\in\RR$.
For $0\le t\le u-1$, let $X_{n,t,u}$ denote the set of triples $(i,j,k)$ such that $1\le i,j,k\le n$ and
at least one of $i$, $j$ and $k$ is congruent to $t$ modulo $u$.  Then each component of $R_n\setminus X_{n,t,u}$
has at most $(u-1)^3$ vertices, and thus $X_{n,t,u}$ belongs to $R_n-\CC$.  If $n\le u-1$, then set $\pi(\emptyset)=1$
and $\pi(X)=0$ for every non-empty $X\in R_n-\CC$.  If $n\ge u$, then set $\pi(X_{n,t,u})=1/u$ for $0\le t\le u-1$
and $\pi(X)=0$ for every other $X\in R_n-\CC$.

Note that every vertex $v$ of $R_n$ belongs to at most three of the sets $X_{n,t,u}$, and thus
the probability that $v$ belongs to a set chosen according to the described distribution is at most $3/u\le\varepsilon$.
Thus, $\pi$ is a fractional $\CC$-complementary packing in $R_n$ of thickness at most $\varepsilon$.

Since such a fractional packing exists for every $\varepsilon>0$ and $R_n\in\RR$, it follows that
$\RR$ is fractionally $\Vs$-fragile.
\end{proof}

Let us remark that it is not a coincidence that $\RR$ is not only fractionally $\Tw$-fragile but also fractionally $\Vs$-fragile,
as we will see in Corollary~\ref{cor-twvs}.

\section{Properties and applications of fractional fragility}\label{sec-fragilityap}

As we have already mentioned in the introduction, we need the notion of fractional $\Vs$-fragility when showing the subexponential
expansion property of graph classes with strongly sublinear separators.  Nevertheless, the notion of fractional $\PP$-fragility
appears to be of independent interest.  We can consider it to be a measure of the distance of the graph class from some property.
Also, many of the algorithmic applications of $\Tw$-fragility also work for the fractional relaxation, which extends them to more
graph classes.

This section is devoted to establishing the basic properties of fractional fragility and showcasing some of its applications.
While this may help the reader to obtain a better understanding of the notion, we do not use these results in the rest of the paper, and
thus the reader may skip to the next section if they prefer to.

Let us first give two examples of algorithmic applications of fractional $\Tw$-fragility (both of which are straightforward generalizations of previously known
results for $\Tw$-fragility).  Of course, in this context we need to be able to find the fractional packings that certify
the fractional $\Tw$-fragility efficiently.
For $c\ge 1$, we say that a class $\GG$ of graphs is \emph{$O(n^c)$-effectively fractionally $\Tw$-fragile} if for
every integer $k\ge 1$, there exists a constant $p_k$, a $(1/k)$-witness $\CC_k$ of the fractional $\Tw$-fragility of $\GG$,
and an algorithm with
\begin{description}
\item[input:] a graph $G\in \GG$, and
\item[output:] a fractional $\CC_k$-complementary packing of thickness at most $1/k$ in $G$,
which assigns a non-zero value to at most $p_k$ elements of $G-\CC_k$,
\end{description}
and the time complexity of the algorithm is $O(|V(G)|^c)$.

The \emph{independence number} $\alpha(G)$ is the size of the largest independent set of a graph $G$.
Determining the independence number of a graph is an NP-complete problem~\cite{garey1979computers}, and even
approximating it up to a polynomial factor is not possible in polynomial time unless $\text{P}=\text{NP}$~\cite{apxindep}.  Nevertheless, a polynomial-time approximation scheme for the
independent set exists for graphs from any $O(n^c)$-effectively fractionally $\Tw$-fragile class of graphs.

\begin{lemma}\label{lemma-indep}
Let $c\ge 1$ and let $\GG$ be an $O(n^c)$-effectively fractionally $\Tw$-fragile class of graphs.
For every $\varepsilon>0$, there exists an algorithm with time complexity $O(|V(G)|^c)$ that for a graph $G\in \GG$ returns an independent set of $G$ of size at
least $(1-\varepsilon)\alpha(G)$.
\end{lemma}
\begin{proof}
Let $k=\lceil 1/\varepsilon\rceil$.
The algorithm first finds a fractional $\CC_k$-complementary packing $\pi$ in $G$ of thickness at most $1/k\le \varepsilon$,
using the algorithm from the definition of $O(n^c)$-effective fractional $\Tw$-fragility.
Let $X_1$, \ldots, $X_p$ be the elements of $G-\CC_k$ to which $\pi$ assigns a non-zero probability, where $p\le p_k$. 
For $1\le i\le p$, let $A_i$ be a largest independent set in $G-X_i$, which can be found in linear time since $\CC_k$ has bounded
treewidth~\cite{courcelle}.  The algorithm returns the largest of $A_1$, \ldots, $A_p$.

Let $A$ be a largest independent set in $G$.
For a set $X\in G-\CC_k$ chosen at random according to the probability distribution $\pi$, each vertex belongs to $X$ with
probability at most $1/k\le \varepsilon$, and thus the expected size of $X\cap A$ is at most
$\varepsilon|A|$.  Hence, there exists $X\in G-\CC$ with $\pi(X)>0$ such that $|X\cap A|\le \varepsilon|A|$.
Since $A\setminus X$ is an independent set in $G-X$, we conclude that
$\alpha(G-X)\ge |A\setminus X|\ge (1-\varepsilon)|A|$.  Therefore, the algorithm indeed returns an independent set of size at least
$(1-\varepsilon)\alpha(G)$.
\end{proof}

Another problem that we consider is testing the existence of a subgraph.  Testing whether a clique $K_n$ is
a subgraph of $G$ is equivalent to verifying that the complement of $G$ has independence number at least $n$,
and thus if the tested subgraph is a part of the input, then the problem is NP-complete.
To test whether a fixed graph $H$ is a subgraph of $G$, we can test all $O\bigl(|V(G)|^{|V(H)|}\bigr)$ choices
for the possible placement of the vertices of $H$ in $G$, or we can use a more involved algorithm
of Ne\v{s}etril and Poljak~\cite{nepol}.  In both cases, we obtain a polynomial-time algorithm
whose exponent depends on $H$, and this cannot be avoided in general unless $\text{FPT}=W[1]$, see~\cite{fellows}.
However, if $G$ is furthermore restricted to belong to an $O(n^c)$-effectively fractionally $\Tw$-fragile class of graphs,
we can design a polynomial-time algorithm whose exponent is independent of $H$.

\begin{lemma}\label{lemma-sg}
Let $c\ge 1$ and let $H$ be a fixed graph.  If a class $\GG$ is $O(n^c)$-effectively fractionally $\Tw$-fragile, then there
exists an algorithm determining whether $H\subseteq G$ for graphs $G\in\GG$ with time complexity $O(|V(G)|^c)$.
\end{lemma}
\begin{proof}
Let $k=|V(H)|+1$.  The algorithm finds a $\CC_k$-complementary packing $\pi$ in $G$ of thickness at most $1/k$,
using the algorithm from the definition of $O(n^c)$-effective fractional $\Tw$-fragility.
Let $X_1$, \ldots, $X_p$ be the elements of $G-\CC_k$ to which $\pi$ assigns a non-zero probability, where $p\le p_k$.      
For $1\le i\le p$, determine whether $H\subseteq G-X_i$ in linear time, since $\CC$ has bounded
treewidth.  If $H$ is a subgraph of one of $G-X_1$, \ldots, $G-X_n$, then $H$ is also a subgraph of $G$.
Otherwise, the algorithm returns that $H$ is not a subgraph of $G$.

Clearly, if $H$ is not a subgraph $G$, then the algorithm correctly determines this.
Suppose that $H$ is a subgraph of $G$, and let $S\subseteq V(G)$ be the set of vertices of this subgraph.
For a set $X\in G-\CC_k$ chosen at random according to the probability distribution $\pi$,
the expected size of $X\cap S$ is at most $|S|/k<1$.  Hence, there exists $X\in G-\CC_k$ with $\pi(X)>0$ such that $X\cap S=\emptyset$, and thus $H\subseteq G-X$.
It follows that the algorithm correctly determines whether $H\subseteq G$.
\end{proof}

In these algorithms, bounded treewidth could be replaced by any other class property which ensures efficient solvability
of the considered problem.  Furthermore, the notion of efficiency could be relaxed,
and we could for instance only require to be able to sample from the probability distribution efficiently (which would
turn the algorithms to probabilistic ones).

In the rest of the text, we do not consider the algorithmic constraint
of being able to find the packings efficiently, and only discuss the graph-theoretic questions concerning fragility and fractional fragility
(although, let us remark that the argument proving Lemma~\ref{lemma-split} can be implemented in an $O(n^c)$-effective way for every $c>2$).
The following is obvious.

\begin{observation}\label{obs-basic}
Let $\GG_1$ and $\GG_2$ be graph classes, and let $\PP_1$ and $\PP_2$ be class properties.
\begin{itemize}
\item If $\GG_1$ is $\PP_1$-fragile, it is also fractionally $\PP_1$-fragile.
\item If $\GG_1$ is (fractionally) $\PP_1$-fragile and $\GG_2\subseteq \GG_1$, then $\GG_2$ is (fractionally) $\PP_1$-fragile.
\item If $\PP_1\subseteq\PP_2$, and $\GG_1$ is (fractionally) $\PP_1$-fragile, then $\GG_1$ is (fractionally) $\PP_2$-fragile.
\end{itemize}
\end{observation}

Fractional fragility is transitive in the following sense (so, for example, if a class $\GG$ is fractionally $\PP$-fragile for a class property $\PP$
whose elements contain only planar graphs, then $\GG$ is also fractionally $\Tw$-fragile).

\begin{lemma}\label{lemma-trans}
Let $\PP_1$ and $\PP_2$ be class properties such that every class in $\PP_1$ is fractionally $\PP_2$-fragile.
If a class $\GG$ is fractionally $\PP_1$-fragile, then it also is fractionally $\PP_2$-fragile.
\end{lemma}
\begin{proof}
Consider any $\varepsilon>0$.  Let $\CC_1$ be a $(\varepsilon/2)$-witness of the fractional $\PP_1$-fragility of $\GG$.
Since $\CC_1\in\PP_1$, the class $\CC_1$ is fractionally $\PP_2$-fragile.
Let $\CC_2$ be an $(\varepsilon/2)$-witness of the fractional $\PP_2$-fragility of $\CC_1$.

Consider a graph $G\in \GG$ and let $\pi_1$ be its fractional $\CC_1$-complementary packing of thickness at most $\varepsilon/2$.
For every $Z\in G-\CC_1$, let $\pi_Z$ be a fractional $\CC_2$-complementary packing of thickness at most $\varepsilon/2$ of $G-Z$.
Let $X\in G-\CC_2$ be chosen at random as follows:  First, select $X_1\in G-\CC_1$ at random according to the distribution
$\pi_1$. Then, select $X_2\in (G-X_1)-\CC_2$ at random according to the distribution $\pi_{X_1}$.  Let $X=X_1\cup X_2$.
This procedure for choosing $X\in G-\CC_2$ defines a probability distribution $\pi$ on $G-\CC_2$.
The probability that a vertex $v\in V(G)$ belongs to $X$ chosen at random according to the distribution $\pi$ is equal to the probability that
either $v\in X_1$, or $v\not\in X_1$ and $v\in X_2$.  Each of these probabilities is at most $\varepsilon/2$, and thus the probability
that $v$ belongs to $X$ is at most $\varepsilon$.
Therefore, $\pi$ is a fractional $\CC_2$-complementary packing in $G$ of thickness at most $\varepsilon$.

Since such a fractional $\CC_2$-complementary packing exists for every $G\in\GG$, it follows that $\CC_2$ is
an $\varepsilon$-witness of fractional $\PP_2$-fragility of $\GG$.
As the choice of $\varepsilon>0$ was arbitrary,
we conclude that $\GG$ is fractionally $\PP_2$-fragile.
\end{proof}

\section{Fragility and bounded expansion}\label{sec-frabe}

Let us now derive the connection to bounded expansion, which we use in the proof of Theorem~\ref{thm-main}.
Let us recall that $\Be$ denotes the class property consisting of all classes with bounded expansion.
For a function $f:\mathbf{N}\to \mathbf{N}$, let $\Ex(f)$ denote the class of all graphs $G$ such that $\nabla_k(G)\le f(k)$ for
all integers $k\ge 0$.
Note that for every fractionally $\Be$-fragile class of graphs $\GG$, there exists a function
$g:\mathbf{R}^+\times \mathbf{N}\to \mathbf{N}$ such that for every $\varepsilon>0$, the class $\Ex(g(\varepsilon,\cdot))$ is an
$\varepsilon$-witness of the fractional $\Be$-fragility of $\GG$.
If $g$ satisfies this property, we say that $\GG$ is \emph{fractionally $(\Be,g)$-fragile}.

As the following lemma shows, a class of graphs is fractionally $\Be$-fragile if and only if it has bounded expansion
(and consequently, if and only if it is $\Be$-fragile), and thus the notion of the fractional $\Be$-fragility does not
bring anything qualitatively new.  Nevertheless, the quantitative relationship between the respective expansion functions
will be of importance later.

\begin{lemma}\label{lemma-bexp}
Let $g:\mathbf{R}^+\times \mathbf{N}\to \mathbf{N}$ be an arbitrary function, and let us define a function $f:\mathbf{N}\to \mathbf{N}$ by setting $f(k)=2g\left(\frac{1}{4k+4},k\right)$.
If $\GG$ is a fractionally $(\Be,g)$-fragile class of graphs, then the expansion of $\GG$ is bounded by $f$.
\end{lemma}
\begin{proof}
Fix $k\ge 0$.  Let $\varepsilon=\frac{1}{4k+4}$ and let $\CC=\Ex(g(\varepsilon,\cdot))$.

Consider an arbitrary graph $G\in\GG$ and let $H$ be a $k$-minor of $G$.  Let $V(H)=\{v_1,\ldots, v_h\}$.  The presence of
$H$ as a $k$-minor of $G$ is certified by vertex-disjoint rooted trees $T_1, \ldots, T_h\subseteq G$ such that
\begin{itemize}
\item for $1\le i\le h$, the tree $T_i$ has depth at most $k$, and
\item if $v_iv_j\in E(H)$, then there exists an edge $e_{ij}\in E(G)$ joining a vertex of $T_i$ with a vertex of $T_j$.
\end{itemize}
For an edge $v_iv_j\in E(H)$, let $P_{ij}$ be the path of length at most $2k+1$ consisting of $e_{ij}$ and
the paths from the ends of $e_{ij}$ to the roots of $T_i$ and $T_j$.

Since $\GG$ is fractionally $(\Be,g)$-fragile, there exists a fractional $\CC$-comple\-men\-tary packing $\pi$ in $G$ of thickness at most $\varepsilon$.
Let $X\in G-\CC$ be chosen at random
according to $\pi$.  Let $H'$ be the subgraph of $H$ consisting of edges $v_iv_j$ such that $P_{ij}$ is disjoint with $X$, and
of the vertices incident with these edges.
The probability that $P_{ij}$ intersects $X$ is at most $\varepsilon|V(P_{ij})|\le 1/2$, and thus the expected number of edges of $H'$ is at least
$|E(H)|/2$.  Let us fix a set $X\in G-\CC$ so that $|E(H')|\ge|E(H)|/2$.  Note that $H'$ is a $k$-minor of $G-X$, and thus
$\frac{|E(H')|}{|V(H')|}\le g(\varepsilon,k)$.  However,
$\frac{|E(H')|}{|V(H')|}\ge \frac{|E(H)|/2}{|V(H')|}\ge \frac{1}{2}\cdot \frac{|E(H)|}{|V(H)|}$.

It follows that $\frac{|E(H)|}{|V(H)|}\le 2g(\varepsilon,k)=f(k)$ for every $k$-minor $H$ of $G$,
and thus $\nabla_k(G)\le f(k)$.  Since this holds for every $G\in \GG$ and every $k\ge 0$, the expansion of $\GG$ is bounded by $f$.
\end{proof}

\section{Fragility and sublinear separators}\label{sec-split}

Next, we study the connection between fractional fragility and sublinear separators.

\begin{lemma}\label{lemma-sublin}
Let $\PP$ be a class property such that every class in $\PP$ has sublinear separators.
If $\GG$ is a fractionally $\PP$-fragile class of graphs, then $\GG$ has sublinear separators.
\end{lemma}
\begin{proof}
Suppose for a contradiction that $\limsup_{n\to\infty} s_\GG(n)/n=\delta>0$.  
Let $\CC$ be a $(\delta/4)$-witness of the fractional $\PP$-fragility of $\GG$; by the assumptions, we have $\lim_{n\to\infty} s_\CC(n)/n=0$. 
Let $n_0\ge 0$ be the smallest integer such that $s_\CC(n)<\frac{\delta}{4}n$ for every $n\ge n_0$.

Since $\limsup_{n\to\infty} s_\GG(n)/n=\delta$, there exists a graph $G\in \GG$ on $n\ge n_0$ vertices such that every balanced separation in $G$
has size at least $\frac{\delta}{2} n$.
Let $\pi$ be a fractional $\CC$-complementary packing of thickness at most $\delta/4$ in $G$.  The expected size of a set chosen at random
according to $\pi$ is at most $\delta n/4$, and thus there exists $X\in G-\CC$ such that $|X|\le \delta n/4$.
Let $(A',B')$ be a balanced separation of $G-X$ of size at most $s_\CC(n)<\frac{\delta}{4}n$.
Let $(A,B)$ be a corresponding separation of $G$ with $A'\subseteq A$, $B'\subseteq B$ and $V(A)\cap V(B)=(V(A')\cap V(B'))\cup X$.
Note that $(A,B)$ is a balanced separation in $G$.
However, $|V(A)\cap V(B)|\le |X|+s_\CC(n)<\frac{\delta}{2}n$, which contradicts the choice of $G$.
\end{proof}

By the previous lemma, having sublinear separators is a necessary condition for fractional $\Vs$-fragility.
Further necessary condition is bounded maximum degree.

\begin{observation}\label{obs-dgneed}
If a class $\GG$ is fractionally $\Vs$-fragile, then it has bounded maximum degree.
\end{observation}
\begin{proof}
Let $\CC$ be a $(1/3)$-witness of the fractional $\Vs$-fragility of $\GG$.
Let $s$ be the maximum size of a component of a graph from $\CC$.

Consider any graph $G\in \GG$, and let $\pi$ be its fractional $\CC$-complementary packing of thickness at most $1/3$.
Let $v$ be any vertex of $G$ and let $N$ be the set of all neighbors of $v$.
The probability that a set $X$ chosen at random according to $\pi$ contains $v$ is at most $1/3$.
Furthermore, the expected size of the intersection of $X$ and $N$ is at most $|N|/3$, and by Lemma~\ref{lemma-markov}, the probability
that $|X\cap N|>|N|/2$ is less than $2/3$.  Therefore, there exists $X\in G-\CC$ such that $v\not\in X$ and $|X\cap N|\le |N|/2$.
Since $G-X$ contains $v$ and at least $|N|/2$ of its neighbors, it has a component of size at least $|N|/2+1$.  Since
every component of $G-X$ has size at most $s$, it follows that $|N|\le 2s-2$.  Therefore, every graph from $\GG$ has maximum degree
at most $2s-2$.
\end{proof}

As the main result of this section, we show that for subgraph-closed graph classes, these necessary conditions are almost sufficient
(we require strongly sublinear separators).  
Let us recall a strong separation property for graphs with bounded tree-width, see~\cite{rs2}.
\begin{lemma}\label{lemma-septw}
For any graph $G$ and a set $X\subseteq V(G)$, there exists a separation $(A,B)$ of $G$ of size at most $\tw(G)+1$ such that
$|X\setminus V(A)|,|X\setminus V(B)|\le 2|X|/3$.
\end{lemma}

We need an auxiliary lemma.

\begin{lemma}\label{lemma-few}
Every graph $G$ has a rooted tree decomposition $(T,\beta)$ with bags of size at most $12(\tw(G)+1)(\Delta(G)+1)$ such that every vertex of $T$ has at most
two sons and for each $v\in V(G)$, the subtree $T[\{u:v\in \beta(u)\}]$ has depth at most $1+4\log(\Delta(G)+1)$.
\end{lemma}
\begin{proof}
Let $b=12(\tw(G)+1)$ and $w=\Delta(G)b$.

Let us consider the following algorithm to obtain a tree decomposition.
\begin{description}
\item[input:] A graph $H$ of maximum degree at most $\Delta(G)$ and tree-width at most $\tw(G)$, and
a set $Z\subseteq V(H)$ of size at most $w$ (which we call the \emph{root set}).
\item[output:] A rooted tree decomposition of $H$ with bags of size at most $w+b$, whose root bag contains $Z$.
\end{description}
\begin{itemize}
\item If $|V(H)|\le w+b$, then let the decomposition consist of a single bag containing all vertices.
\item If $|Z|\le b$, then let $Z'$ be a superset of $Z$ of size $b$ and let $Z''$ consist of all vertices of $V(H)\setminus Z'$ that have a neighbor in $Z'$.
Let us apply the algorithm recursively to $H-Z'$ with the root set $Z''$ (note that this is possible,
since $|Z''|\le \Delta(H)|Z'|\le w$).  To the root of the resulting decomposition, attach a father node
whose bag is $Z'\cup Z''$.
\item Otherwise, let $(A,B)$ be a separation of $H$ of size at most $\tw(H)+1$ with $|Z\setminus V(A)|,|Z\setminus V(B)|\le 2|Z|/3$
which exists by Lemma~\ref{lemma-septw}.
Let $Z_A=(V(A)\cap V(B))\cup (Z\setminus V(B))$ and $Z_B=(V(A)\cap V(B))\cup (Z\setminus V(A))$.  Apply the algorithm recursively to
$A$ with the root set $Z_A$, and to $B$ with the root set $Z_B$ (this is possible, since $|Z_A|,|Z_B|<|Z|\le w$ as we argue below).
Add a common father node of their roots with bag $Z\cup V(A\cap B)$ to the resulting tree decomposition.
\end{itemize}

To obtain the required tree decomposition of $G$, we run the described algorithm for $G$ with the empty root set.
Note that the case that $|Z|>b$ can only be reached if $\Delta(G)\ge 2$, and that in this case $|Z_A|,|Z_B|\le 2|Z|/3+\tw(G)+1<3|Z|/4$.
Therefore, after at most $\lceil\log(w/b)/\log(4/3)\rceil\le 4\log(\Delta(G)+1)$ levels of recursion, we reach the case that $|Z|\le b$,
and all the vertices of $Z$ are excluded from the graph in the next recursive call.
Hence, every vertex appears in the bags of at most $2+4\log(\Delta(G)+1)$ consecutive levels of the tree decomposition.
\end{proof}

It is easy to see that if every subgraph of an $n$-vertex graph $G$ has a balanced separator of size at most $b$, then
$G$ has treewidth $O(b\log n)$.  Recently, a stronger claim was proved (the weaker bound with the logarithmic
factor would suffice for the purposes of this paper, however using Theorem~\ref{thm-dnorin} simplifies the computations
a bit).

\begin{theorem}[Dvo\v{r}\'ak and Norin~\cite{dnorin}]\label{thm-dnorin}
If $\GG$ is a subgraph-closed class of graphs, then every graph $G\in\GG$ has treewidth
at most $105s_\GG(|V(G)|)$.
\end{theorem}

We are now ready to prove the main result of this section, which is a reformulation of Lemma~\ref{lemma-split}.

\begin{lemma}\label{lemma-sublind}
Let $\GG$ be subgraph-closed class of graphs.  If $\GG$ has bounded maximum degree and strongly sublinear separators,
then $\GG$ is fractionally $\Vs$-fragile.  Furthermore, there exists a constant $b>1$ such that for every $0<\varepsilon\le 1$,
the class $\CC_\varepsilon$ of all graphs in $\GG$ such that all their components have at most $b^{1/\varepsilon}$ vertices is an $\varepsilon$-witness
of the fractional $\Vs$-fragility of $\GG$.
\end{lemma}
\begin{proof}
Let $c\ge 1$ and $0\le \delta<1$ be real numbers such that $s_\GG(n)\le cn^\delta$ for every $n\ge 0$.
Let $\Delta\ge 0$ be an integer such that every graph in $\GG$ has maximum degree at most $\Delta$.  Let $\iota>0$ be chosen arbitrarily so that
$\delta+\iota<1$.  Let $c_1=105c$, so that every graph in $\GG$ on $n$ vertices has treewidth at most $c_1n^\delta$ by Theorem~\ref{thm-dnorin}.
Let $c_2=24(c_1+1)(\Delta+1)$, $c'_2=c_2^{\frac{1}{1-(\delta+\iota)}}$, $c_3=\frac{1}{\delta+\iota}>1$, $c_4=\frac{2+4\log(\Delta+1)}{(c_3-1)\iota}$,
$c_5=e^{c_3c_4}$ and $b=c'_2c_5$.

We first prove the following auxiliary claim:  
\begin{itemize}
\item[($\star$)]\emph{Let $G$ be a graph in $\GG$ with at most $n$ vertices, and let
$$S=\{X\subseteq V(G):\mbox{every component of $G-X$ has at most $c_2n^{\delta+\iota}$ vertices}\}.$$
There exists a probability distribution $\pi_{G,n}$ on $S$ such that every vertex of $G$ has probability at most
$$\frac{2+4\log(\Delta+1)}{\iota\log n}$$
of appearing in a set chosen according to this distribution.}
\end{itemize}

By the choice of $c_1$, the graph $G$ has tree-width at most $c_1n^\delta$.
Let $(T,\beta)$ be the rooted tree decomposition of $G$ obtained using Lemma~\ref{lemma-few}.
Recall that each vertex of $T$ has at most two sons, and note that each bag of the decomposition has
size at most $12(c_1n^\delta+1)(\Delta+1)\le 12(c_1+1)(\Delta+1)n^\delta=\frac{c_2}{2}n^\delta$.

Let $k=\lceil \iota\log n\rceil$ and for $0\le i\le k-1$,
let $X_i$ be the set of vertices of $G$ which appear in the bags of the decomposition whose distance $d$ from the root satisfies
$d\equiv i\pmod{k}$.
Consider any connected component $H$ of $G-X_i$.  Let $T_H$ be the subtree of $T$ induced by $\{u:\beta(u)\cap V(H)\neq\emptyset\}$
and let $\beta_H:V(T_H)\to 2^{V(H)}$ be defined by $\beta_H(u)=\beta(u)\cap V(H)$ for $u\in V(T_H)$.
Then $(T_H,\beta_H)$ is a rooted tree decomposition of $H$ such that each vertex of $T_H$ has at most two sons
and each bag of $T_H$ has size at most $\frac{c_2}{2}n^\delta$.  Furthermore, by the choice of $X_i$, the tree
$T_H$ has depth at most $k-2$, hence $|V(T_H)|\le 2^k\le 2^{\iota\log n+1}\le 2n^\iota$ and 
$|V(H)|\le \frac{c_2}{2}n^\delta|V(T_H)|\le c_2n^{\delta+\iota}$.
Since this holds for every connected component of $G-X_i$, the set $X_i$ belongs to $S$.

For every $X\in S$, let $\pi_{G,n}(X)=\frac{|\{i:0\le i\le k-1, X_i=X\}|}{k}$.
Since we chose $(T,\beta)$ using Lemma~\ref{lemma-few}, each vertex of $G$ appears
in at most $2+4\log(\Delta+1)$ of the sets $X_0$, \ldots, $X_{k-1}$, and
thus the probability that a vertex appears in a set chosen
according to the distribution $\pi_{G,n}$ is at most $\frac{2+4\log(\Delta+1)}{k}\le \frac{2+4\log(\Delta+1)}{\iota\log n}$.
This finishes the proof of ($\star$).

\bigskip

We now iterate this construction for an $n$-vertex graph $G\in\GG$.
Let $n_0=n$ and $n_{i+1}=c_2n_i^{\delta+\iota}$ for $i\ge 0$.  Note that
\begin{equation}\label{eq-ni}
n_i\le c_2^{1+(\delta+\iota)+(\delta+\iota)^2+\ldots}n^{(\delta+\iota)^i}=c_2^{\frac{1}{1-(\delta+\iota)}}n^{(\delta+\iota)^i}=c'_2n^{(\delta+\iota)^i}=c'_2n^{1/c_3^i}
\end{equation}
and since $c_2\ge 1$,
\begin{equation}\label{eq-nilb}
n_i\ge n^{({\delta+\iota})^i}=n^{1/c_3^i}.
\end{equation}
Consider a graph $G\in\GG$ with $n$ vertices.  Let $G_0=G$.
For $i\ge 0$, $G_i$ will be some subgraph of $G$ such that each component of $G_i$ has size at most $n_i$.
Note that since $\GG$ is subgraph-closed, every component of $G_i$ belongs to $\GG$.
To construct $G_{i+1}$, for each component $G'_i$ of $G_i$ consider the probability distribution $\pi_{G'_i, n_i}$ obtained in ($\star$) and choose a subset
of $V(G'_i)$ at random according to this distribution (independently in each component).  Let $X_i$ be the union of all these subsets and let $G_{i+1}=G_i-X_i$.

For an integer $t\ge 0$, let $Y_t=X_0\cup X_1\cup \ldots\cup X_{t-1}$.  Note that $Y_t$ is a subset of $V(G)$ chosen at random
according to a probability distribution described by the construction of the previous paragraph, and that
each component of $G-Y_t$ has size at most $n_t$.
The probability that a vertex $v\in V(G)$ belongs to this set $Y_t$
is at most
\begin{equation}\label{eq-pry}
\sum_{i=0}^{t-1} \text{Prob}[v\in X_i]\le \sum_{i=0}^{t-1} \frac{2+4\log(\Delta+1)}{\iota\log n_i}=\frac{2+4\log(\Delta+1)}{\iota}\sum_{i=0}^{t-1} \frac{1}{\log n_i}
\end{equation}
by ($\star$).  By (\ref{eq-nilb}) we have $\log n_i\ge \frac{1}{c_3^i}\log n$, and thus

$$\sum_{i=0}^{t-1} \frac{1}{\log n_i}\le \frac{1}{\log n}\sum_{i=0}^{t-1}c_3^i=\frac{c_3^t-1}{(c_3-1)\log n}\le\frac{c_3^t}{(c_3-1)\log n}.$$
Combined with (\ref{eq-pry}), this implies that the probability that $v$ belongs to $Y_t$ is at most
\begin{equation}\label{eq-pryt}
\frac{2+4\log(\Delta+1)}{\iota}\sum_{i=0}^{t-1} \frac{1}{\log n_i}\le \frac{(2+4\log(\Delta+1))c_3^t}{\iota(c_3-1)\log n}=\frac{c_4c_3^t}{\log n}.
\end{equation}

Consider any $\varepsilon$ such that $0<\varepsilon\le 1$.  Recall that $\CC_\varepsilon$ is the class of all graphs in $\GG$ such that all their components have at most $b^{1/\varepsilon}$ vertices.

If $\log n<c_4/\varepsilon$, then $|V(G)|=n=e^{\log n}<e^{c_4/\varepsilon}<b^{1/\varepsilon}$, and thus 
$G\in \CC_\varepsilon$ and setting $\pi(\emptyset)=1$ and $\pi(X)=0$ for every non-empty $X\subseteq V(G)$ gives a fractional $\CC_\varepsilon$-complementary packing in
$G$ of thickness $0$.

Suppose now that $\log n\ge c_4/\varepsilon$.
Fix $t\ge 0$ as the largest integer such that $c_4c_3^t/\log n\le \varepsilon$, and let $\pi$ be the probability distribution on the subsets
of $V(G)$ described by the process generating the set $Y_t$ (so $Y_t$ is chosen at random from the distribution $\pi$).
By the maximality of $t$, we have $c_3^{t+1}>\frac{\varepsilon\log n}{c_4}$, and $c_3^t>\frac{\varepsilon\log n}{c_3c_4}$.
As we observed before, each component of $G-Y_t$ has size at most $n_t$.  However, by (\ref{eq-ni}),
$$n_t\le c'_2n^{1/c_3^t}\le c'_2n^{\frac{c_3c_4}{\varepsilon\log n}}=c'_2e^{\frac{c_3c_4}{\varepsilon}}=c'_2c_5^{1/\varepsilon}\le b^{1/\varepsilon},$$
and thus $G-Y_t\in \CC_\varepsilon$.
By (\ref{eq-pryt}), the probability that a vertex $v\in V(G)$ belongs to
$Y_t$ is at most $\frac{c_4c_3^t}{\log n}\le \varepsilon$.
Therefore, $\pi$ is a fractional $\CC_\varepsilon$-complementary packing in $G$ of thickness at most $\varepsilon$.

Since this holds for every $\varepsilon>0$, we conclude that $\GG$ is fractionally $\Vs$-fragile and that $\CC_\varepsilon$ is an $\varepsilon$-witness of the fractional $\Vs$-fragility of $\GG$.
\end{proof}

Let us remark that for every $k\ge 0$, the class of all graphs with treewidth at most $k$ has strongly sublinear separators.  Hence,
Lemma~\ref{lemma-sublind} together with Lemma~\ref{lemma-trans} implies the following.

\begin{corollary}\label{cor-twvs}
Every fractionally $\Tw$-fragile class with bounded maximum degree is fractionally $\Vs$-fragile.
\end{corollary}

\section{Expansion in small graphs}\label{sec-small}

Roughly, the aim of this section is to show that for $\varepsilon>0$, if a sufficiently large $n$-vertex graph $G$ satisfies $\nabla_{\lceil \log n\rceil}(G)\ge n^\varepsilon$,
then $G$ contains a subgraph without a strongly sublinear separation.  To do so, we show that for some $\varepsilon'>0$, $G$ contains a shallow minor of a clique $K_s$
with $s\ge n^{\varepsilon'}$.  This clique contains a $3$-regular expander on $s$ vertices as a subgraph, and thus $G$ contains a shallow subdivision
of this $3$-regular expander.  It is easy to see that such a subdivision does not have strongly sublinear separators.

Consider a $\lceil\log n\rceil$-minor $G'$ of $G$ with edge density $\nabla_{\lceil \log n\rceil}(G)\ge n^\varepsilon$.  By Koml\'os and Szemer\'edi~\cite{komszem2} and Thomasson~\cite{minor2},
$G'$ contains $K_{\lfloor n^{\varepsilon/2}\rfloor}$ as a minor (actually, a topological minor).  However, their proofs give no bound
on the depth of the minor.  Topological minors with edges subdivided bounded number of times were studied by Kostochka and Pyber~\cite{kopyb}
and Jiang~\cite{jiang}, however their results do not give polynomially large cliques.  Hence, we need to derive a result combining both shallowness
and polynomial size.  Let us remark that doing so in the terms of topological minors is possible~\cite{dvorak}, however it will be more convenient to
only give the result for minors.

We are going to repeatedly take shallow minors, and the following observation will be useful.

\begin{observation}\label{obs-shallow}
If $H'$ is a $d_1$-minor of $G$ and $H$ is a $d_2$-minor of $H'$, then $H$ is a $(d_1+d_2(2d_1+1))$-minor of $G$.
\end{observation}

We also use the following result, which gives shallow minors in very dense graphs.
Let $K'_t$ denote the graph obtained from $K_t$ by subdividing each edge by exactly one vertex.

\begin{lemma}[{Jiang~\cite[Proposition 2.3]{jiang}}]\label{lemma-jiang}
For any $t\ge 1$, if a graph $G$ on $n$ vertices has at least $t^2n^{3/2}$ edges, then $G$ contains $K'_t$ as a subgraph.
\end{lemma}

For sparser graphs, we use the following lemma to find denser $1$-minors.

\begin{lemma}\label{lemma-denser}
Suppose that $c\ge 64$, $t\ge 1$ and $0<\varepsilon<1$.  If a graph $G$ on $n$ vertices has at least $ct^4n^{1+\varepsilon}$ edges,
then it contains either a graph $G'$ with at least $\frac{c}{32}t^4|V(G')|^{1+\varepsilon+\varepsilon^2}$ edges as a $1$-minor,
or $K_t$ as a $4$-minor.
\end{lemma}
\begin{proof}
Note that removing a vertex of degree at most $ct^4n^\varepsilon$ from $G$ results in a graph on $n-1$ vertices and with
at least $ct^4(n-1)n^{\varepsilon}\ge ct^4(n-1)^{1+\varepsilon}$ edges; hence, without loss of generality we can assume that the minimum
degree of $G$ is at least $ct^4n^\varepsilon$.  Consequently, we have 
\begin{equation}\label{eq-lown}
n\ge ct^4n^\varepsilon\text{ and }n^{1-\varepsilon}\ge ct^4.
\end{equation}

Let $A\subseteq V(G)$ be chosen so that the number of edges of $G$ with exactly one end in $A$ is as large as possible, let $B=V(G)\setminus V(A)$
and let $G_1$ be the spanning bipartite subgraph of $G$ consisting of edges of $G$ with exactly one end in $A$.
Consider any vertex $v\in V(G)$, and let $A_v$ be the symmetric difference of $A$ and $\{v\}$.
Observe that $G$ contains at least $|E(G_1)|-\deg_{G_1}(v)+(\deg_G(v)-\deg_{G_1}(v))$ edges with exactly one end in $A_v$,
and by the choice of $A$, it follows that $\deg_{G_1}(v)\ge\frac{1}{2}\deg_G(v)$.
Hence, every vertex of $G_1$ has degree at least $\frac{c}{2}t^4n^\varepsilon$.
By symmetry, we can assume that $|A|\le n/2\le|B|$.

Let $p=n^{-\varepsilon}$.  Let $A'$ be a subset of $A$ obtained by choosing each vertex independently at random with probability $p$.
Since $|A|\le n/2$,  Lemma~\ref{lem-chernoff} with $r=n/2$ and (\ref{eq-lown}) implies that the probability that $|A'|\ge 2rp=n^{1-\varepsilon}$ is less than
\begin{equation}\label{eq-pra}
\exp\Bigl(-\frac{3n^{1-\varepsilon}}{16}\Bigr)\le\exp\Bigl(-\frac{3ct^4}{16}\Bigr)<1/3.
\end{equation}

Consider a vertex $v\in B$.  The expected number of neighbors of $B$ in $A'$ is at least $\delta(G_1)p\ge \frac{c}{2}t^4$, and by Lemma~\ref{lem-chernoff},
the probability that the number of neighbors of $B$ in $A'$ is less than $\frac{c}{4}t^4$ is at most
$$\exp\Bigl(-\frac{3ct^4}{56}\Bigr)<1/3.$$
Let $B_1$ consist of the vertices of $B$ with less than $\frac{c}{4}t^4$ neighbors in $A'$; the expected size of $B_1$ is at most $|B|/3$.
By Lemma~\ref{lemma-markov}, the probability that $|B_1|\ge |B|/2$ is at most $2/3$.  Let $B'=B\setminus B_1$ consist of the vertices of $B$ with
at least $\frac{c}{4}t^4$ neighbors in $A'$; hence,
\begin{equation}\label{eq-prb}
\text{Prob}[|B'|\le |B|/2]\le 2/3.
\end{equation}

By (\ref{eq-pra}) and (\ref{eq-prb}), we have $|A'|<n^{1-\varepsilon}$ and $|B'|>|B|/2\ge n/4$ with non-zero probability;
let us fix a set $A'\subset A$ of size less than $n^{1-\varepsilon}$ such that the set $B'\subseteq B$ of vertices of $B$
with at least $\frac{c}{4}t^4$ neighbors in $A'$ has size greater than $n/4$.

We now form a $1$-minor $G'$ of $G_1$ as follows.  Let the vertex set of $G'$ be $A'$ and initially, let the edge set of $G'$ be empty.
We process each vertex $v\in B'$ in turn.  Let $N_v$ be the set of neighbors of $v$ in $A'$.  If $G'[N_v]$ has a vertex $w$ of degree
at most $|N_v|/2-1$, then we add edges between $w$ and all other vertices of $N_v$ to $G'$ (this corresponds to contracting the edge $vw$ of $G_1$)
and continue processing the remaining vertices of $B'$.  Otherwise, the construction stops.
Note that $G'$ is obtained from $G_1$ by contracting edges of a star forest with centers of the stars contained in $A'$.

Let us consider the case that the construction stops while processing a vertex $v\in B'$.  In that case, $G'[N_v]$ has minimum degree
at least $(|N_v|-1)/2$.  Let $N=N_v\cup \{v\}$ and let $H$ be the graph consisting of $G'[N_v]$ and of the vertex $v$ adjacent to all
vertices of $N_v$.  Observe that $H$ has minimum degree at least $|N|/2$, and that $H$ is a $1$-minor of $G$.
By the choice of $B'$, we have $|V(H)|=|N|>\frac{c}{4}t^4$.
The number of edges of $H$ is at least $|N|^2/4=\frac{\sqrt{|N|}}{4}|N|^{3/2}\ge \frac{\sqrt{c}}{8}t^2|N|^{3/2}\ge t^2|N|^{3/2}$ edges,
since $c\ge 64$.
By Lemma~\ref{lemma-jiang}, $H$ contains $K'_t$ as a subgraph.  By Observation~\ref{obs-shallow}, this gives a $4$-minor of $K_t$ in $G$.
Hence, the second outcome of Lemma~\ref{lemma-denser} holds.

Finally, suppose that all vertices of $B'$ are processed.  Recall that each vertex of $B'$ has at least
$\frac{c}{4}t^4$ neighbors in $A'$; hence, for each vertex of $B'$, we added at least $\frac{c}{8}t^4$ edges to $G'$.
It follows that $G'$ has at least $\frac{c}{8}t^4|B'|\ge \frac{c}{32}t^4n$ edges.
Let us recall that $|V(G')|=|A'|<n^{1-\varepsilon}$, and thus $n\ge|V(G')|^{1/(1-\varepsilon)}$.
Consequently, $G'$ is a $1$-minor of $G$ with at least
$$\frac{c}{32}t^4n\ge \frac{c}{32}t^4|V(G')|^{1/(1-\varepsilon)}=\frac{c}{32}t^4|V(G')|^{1+\varepsilon+\frac{\varepsilon^2}{1-\varepsilon}}\ge\frac{c}{32}t^4|V(G')|^{1+\varepsilon+\varepsilon^2}$$
edges, as required in the first outcome of Lemma~\ref{lemma-denser}.
\end{proof}

Let us remark that since Lemma~\ref{lemma-jiang} gives a $1$-subdivision rather than a $1$-minor, we can improve
the second outcome of Lemma~\ref{lemma-denser} to obtain a $3$-minor of $K_t$ rather than a $4$-minor,
and similarly we could improve other bounds in this section.  Nevertheless, the bounds that we obtain in this way
still are not likely to be close to optimal, and thus we do not go through this extra effort to improve them.

We now iterate Lemma~\ref{lemma-denser}, using Observation~\ref{obs-shallow}.

\begin{corollary}\label{cor-iter}
Suppose that $m\ge 1$, $t\ge 1$ and $0<\varepsilon<1$.
If a graph $G$ on $n$ vertices has at least $2\cdot32^mt^4n^{1+\varepsilon}$ edges,
then it contains either a graph $G'$ with at least $t^4|V(G')|^{1+\varepsilon+m\varepsilon^2}$ edges as a $4^{m-1}$-minor,
or $K_t$ as a $4^m$-minor.
\end{corollary}

By Lemma~\ref{lemma-jiang} and Observation~\ref{obs-shallow}, the first outcome of Corollary~\ref{cor-iter} implies
the second one when $1+\varepsilon+m\varepsilon^2\ge 3/2$.  Hence, we obtain the following.

\begin{corollary}\label{cor-iter2}
Suppose that $0<\varepsilon\le 1$ and let $m=\left\lceil\frac{1}{2\varepsilon^2}\right\rceil$.
If a graph on $n$ vertices has at least $2\cdot 32^mt^4n^{1+\varepsilon}$ edges,
then it contains $K_t$ as a $4^m$-minor.
\end{corollary}

Theorem~\ref{thm-lacli} is then easily obtained by carefully choosing the parameters.
\begin{proof}[Proof of Theorem~\ref{thm-lacli}]
Let $m=\left\lceil\frac{18}{\varepsilon^2}\right\rceil$, $d=4^m$ and let $n_0\ge 0$ be the smallest integer
such that $n_0^{\varepsilon/6}\ge 2\cdot 32^m$.  Let $t=\lfloor n^{\varepsilon/6}\rfloor$.  Then $G$ has at least
$n^{1+\varepsilon}\ge 2\cdot 32^mt^4n^{1+\varepsilon/6}$ edges and the result follows from Corollary~\ref{cor-iter2}.
\end{proof}

To establish subexponential bounds on expansion, we need another consequence of Corollary~\ref{cor-iter2}.

\begin{theorem}\label{thm-iter3}
Suppose that $2/3<\delta\le 1$, $0<\mu\le 1$ and $b>1$.
There exists $k_0\ge 0$ such that for every $k\ge k_0$,
if $t=\bigl\lfloor e^{\frac{1}{6}k^\delta}\bigr\rfloor$ and $G$ is a graph with $n\le b^k$ vertices and with at least $e^{k^\delta}n$ edges, then
$G$ contains $K_t$ as a $t^\mu$-minor.
\end{theorem}
\begin{proof}
Let $\varepsilon=\frac{1}{6k^{1-\delta}\log b}$, $m=\left\lceil\frac{1}{2\varepsilon^2}\right\rceil$ and $k'_0=(\log 64)^{1/\delta}$.
Note that
\begin{equation}\label{eq-upm}
2\cdot 32^m\le 64\cdot 32^{\frac{1}{2\varepsilon^2}}=64e^{(18\log 32\log^2b)k^{2-2\delta}}.
\end{equation}
Furthermore, for $k\ge k'_0$, we have
\begin{align}
e^{\frac{1}{6}k^\delta}&\ge 2\nonumber\\
t=\bigl\lfloor e^{\frac{1}{6}k^\delta}\bigr\rfloor&\ge e^{\frac{1}{6}k^\delta}-1\ge \frac{1}{2}e^{\frac{1}{6}k^\delta}\ge 1\nonumber\\
t^\mu&\ge \frac{1}{2^\mu}e^{\frac{\mu}{6}k^\delta}.\label{eq-lowtmu}
\end{align}
Choose $k_0\ge k'_0$ large enough that $t^\mu\ge 2\cdot 32^m$
for every $k\ge k_0$; this is possible by (\ref{eq-upm}) and (\ref{eq-lowtmu}), since $\delta > 2 - 2\delta$.

Since $n\le b^k$, we have $$n^\varepsilon\le b^{\varepsilon k}=b^{\frac{k}{6k^{1-\delta}\log b}}=e^{k^\delta/6},$$
and thus $$|E(G)|\ge e^{k^\delta}n\ge t^5e^{k^\delta/6}n\ge t^5n^{1+\varepsilon}\ge t^\mu t^4n^{1+\varepsilon}\ge 2\cdot 32^mt^4n^{1+\varepsilon}.$$
By Corollary~\ref{cor-iter2}, $G$ contains $K_t$ as a $4^m$-minor.
Note that $4^m<2\cdot 32^m\le t^\mu$ by the choice of $k_0$, and thus $G$ contains $K_t$ as a $t^\mu$-minor as required.
\end{proof}

\section{Sublinear separators and expansion}\label{sec-subexpexp}

Now, let us turn our attention to sublinear separators.
For $\alpha>0$, a graph $G$ is an \emph{$\alpha$-expander} if for every $S\subseteq V(G)$ of size at most $|V(G)|/2$, there
exist at least $\alpha|S|$ edges of $G$ with exactly one end in $S$.  Random graphs are asymptotically almost surely expanders.

\begin{lemma}[Bollob\'as~\cite{bolo}]\label{lemma-exexp}
There exists $n_0$ such that for every even $n\ge n_0$, there exists a $3$-regular $\frac{3}{20}$-expander on $n$ vertices.
\end{lemma}

Clearly, expanders do not have sublinear-size separations.  This can be extended to their bounded-depth subdivisions.

\begin{lemma}\label{lemma-expsep}
Let $\alpha>0$ be a real number and let $n,m\ge 1$ be integers.  Let $G'$ be obtained from a $3$-regular $\alpha$-expander $G$
on $n$ vertices by subdividing each edge at most $m$ times, and let $n'=|V(G')|$.  Any balanced separation in $G'$ has size at least $\frac{n'}{3(1+3m/2)(6/\alpha+2)}$.
\end{lemma}
\begin{proof}
Let $(A',B')$ be a balanced separation in $G'$ and let $S'=V(A')\cap V(B')$.  Note that
$n'=|V(A')|+|V(B')\setminus S'|\le |V(A')|+2n'/3$, and
thus $|V(A')|\ge n'/3$, and similarly $|V(B')|\ge n'/3$.
For each $v\in V(G)$, let $\tilde{v}$ denote the corresponding vertex of $G'$.
Let $S$ be a minimal subset of $V(G)$ such that 
\begin{itemize}
\item for each $v\in V(G)$, if $\tilde{v}\in S'$, then $v\in S$, and
\item for every path $P'\subseteq G'$ corresponding to an edge $uv\in E(G)$
such that an internal vertex of $P'$ belongs to $S'$, we have $\{u,v\}\subseteq S$.
\end{itemize}
Let $A=\{v\in V(G)\setminus S: \tilde{v}\in V(A')\}$ and $B=\{v\in V(G)\setminus S: \tilde{v}\in V(B')\}$.
Note that $|S|\le 2|S'|$, and that for each two vertices that are connected by a path in $G-S$, the corresponding vertices
are also connected by a path in $G'-S'$.  Consequently, no vertex of $A$ has a neighbor in $B$.  Without loss of generality, we can assume
$|A|\le n/2$, and since $G$ is an $\alpha$-expander, it contains at least $\alpha|A|$ edges with one end in $A$ and the other end in $S$.  Since $G$ is $3$-regular,
we have
\begin{equation}\label{eq-sza}
\alpha|A|\le 3|S|\le 6|S'|.
\end{equation}

Consider a vertex $z\in V(A')$.  If $z=\tilde{v}$ for some $v\in V(G)$, then we have $v\in A\cup S$.  Similarly,
if $z$ is an internal vertex of a path $P'\subseteq G'$ corresponding to an edge $uv\in E(G)$, then $\{u,v\}\subseteq A\cup S$,
as otherwise an end of $P'$ would belong to $V(B')\setminus S'$ and $P'$ would contain an internal vertex belonging to $S'$,
contradicting the choice of $S$.  Let $H$ be the subgraph of $G$ induced by $A\cup S$.  We observed that each vertex
of $A'$ either corresponds to a vertex of $H$, or it is contained in a path of $G'$ replacing an edge of $H$.
Since $H$ has maximum degree at most $3$, it follows that
$$|V(A')|\le |V(H)|+m|E(H)|\le |V(H)|+\frac{3}{2}m|V(H)|=(1+3m/2)|A\cup S|.$$

Since $|A\cup S|=|A|+|S|\le (6/\alpha+2)|S'|$ by (\ref{eq-sza}),
we have $(1+3m/2)(6/\alpha+2)|S'|\ge (1+3m/2)|A\cup S|\ge |V(A')|\ge n'/3$.
Therefore,
$$|S'|\ge \frac{n'}{3(1+3m/2)(6/\alpha+2)},$$
which gives the lower bound on the size of balanced separations in $G'$.
\end{proof}

We are now ready to bound the expansion in small graphs.

\begin{lemma}\label{lemma-smallexp}
Let $\GG$ be a subgraph-closed class of graphs with strongly sublinear separators,
and let $b>1$ be a real number.  There exists $k_0\ge 0$ such that for every $k\ge k_0$, every graph $G\in \GG$ with
at most $b^k$ vertices satisfies $\nabla_k(G)<e^{k^{3/4}}$.
\end{lemma}
\begin{proof}
Let $c>0$ and $0\le\psi<1$ be constants such that $s_\GG(n)\le cn^\psi$ for every $n\ge 0$.
Let $\alpha=\frac{3}{20}$ and let $n_0$ be a constant such that for every even $n\ge n_0$, there exists a $3$-regular $\alpha$-expander on $n$ vertices
(the constant $n_0$ exists by Lemma~\ref{lemma-exexp}).
Let $k_0\ge 1$ be large enough so that Theorem~\ref{thm-iter3} applies with $\delta=3/4$, $\mu=\frac{1-\psi}{2}$ and $b$;
and furthermore, so that any $k\ge k_0$ and $t=\bigl\lfloor e^{\frac{1}{6}k^{3/4}}\bigr\rfloor$
satisfies $t\ge n_0+1$ and $(t-1)^{1-\psi}>126c(1+12kt^\mu)$.

Suppose that for some $k\ge k_0$, there exists $G\in\GG$ with at most $b^k$ vertices satisfying $\nabla_k(G)\ge e^{k^{3/4}}$.
Let $G_1$ be a $k$-minor of $G$ with $n_1$ vertices and at least $e^{k^{3/4}}n_1$ edges.
Note that $n_1\le |V(G)|\le b^k$.
By Theorem~\ref{thm-iter3}, $G_1$ contains $K_t$ as a $t^\mu$-minor.
By Observation~\ref{obs-shallow}, $G$ contains $K_t$ as a $(k+t^\mu(2k+1))$-minor, and thus also as a $4kt^\mu$-minor. 
Let $G_2$ be a $3$-regular $\frac{3}{20}$-expander
with either $t-1$ or $t$ vertices, which exists by Lemma~\ref{lemma-exexp}.  Note that $G_2$ is a $4kt^\mu$-minor of $G$, and since $G_2$
is $3$-regular, there exists a graph $G_3\subseteq G$ obtained from $G_2$ by subdividing each edge at most $8kt^\mu$ times.
Since $\GG$ is subgraph-closed, $G_3$ has a balanced separation of size at most $c|V(G_3)|^\psi$.  On the other hand,
Lemma~\ref{lemma-expsep} implies that every balanced separation in $G_3$ has size at least $\frac{|V(G_3)|}{3(1+3(8kt^\mu)/2)(6/\alpha+2)}=\frac{|V(G_3)|}{126(1+12kt^\mu)}$.
We conclude that
$$c|V(G_3)|^\psi\ge \frac{|V(G_3)|}{126(1+12kt^\mu)},$$
and thus
$$126c(1+12kt^\mu)\ge |V(G_3)|^{1-\psi}\ge (t-1)^{1-\psi}.$$
This is a contradiction by the choice of $k_0$.
\end{proof}

Finally, we combine Lemmas~\ref{lemma-sublind}, \ref{lemma-bexp} and \ref{lemma-smallexp} to obtain the proof of the main result.

\begin{proof}[Proof of Theorem~\ref{thm-main}]
By Lemma~\ref{lemma-sublind}, $\GG$ is fractionally $\Vs$-fragile, and there exists a constant $b_0>1$ such that for any
$0<\varepsilon\le 1$, the class $\CC_\varepsilon$ of all graphs in $\GG$ such that all their components have at most $b_0^{1/\varepsilon}$ vertices is an $\varepsilon$-witness
of the fractional $\Vs$-fragility of $\GG$.  Let $b=b_0^8$.  Let $k_0\ge 1$ be the constant of Lemma~\ref{lemma-smallexp} applied for $\GG$ and $b$.
Let $\gamma=2b^{k_0}$.

For $0<\varepsilon\le 1$ and $k\ge 0$, let $g(\varepsilon,k)=\nabla_k(\CC_\varepsilon)$.  Since $\CC_\varepsilon$ has bounded
component size, it has bounded expansion, and thus $g(\varepsilon,k)$ is finite.  Since for every $0<\varepsilon\le 1$,
every graph in $\GG$ has a fractional $\CC_\varepsilon$-complementary packing of thickness at most $\varepsilon$, the choice of $g$ implies
that $\GG$ is fractionally $(\Be,g)$-fragile.

Let $G'$ be any graph of $\CC_{1/(4k+4)}$. Note that each component of $G'$ has at most $b_0^{4k+4}\le b^k$ vertices, and thus if $k\le k_0$,
we have $\nabla_k(G')\le b^{k_0}$.  On the other hand, if $k>k_0$, then Lemma~\ref{lemma-smallexp} (applied to each component separately)
implies that $\nabla_k(G')<e^{k^{3/4}}$.  In both cases, we have $\nabla_k(G')\le b^{k_0}e^{k^{3/4}}$.
It follows that $$g\left(\frac{1}{4k+4},k\right)=\nabla_k\left(\CC_{1/(4k+4)}\right)\le b^{k_0}e^{k^{3/4}}.$$

By Lemma~\ref{lemma-bexp}, we have $$\nabla_k(\GG)\le 2g\left(\frac{1}{4k+4},k\right)\le 2b^{k_0}e^{k^{3/4}}=\gamma e^{k^{3/4}}$$
for every integer $k\ge 0$, as required.
\end{proof}

\bibliographystyle{siam}
\bibliography{twd}

\end{document}